\documentclass{amsart}
\usepackage{amsfonts}

\newtheorem{theorem}{Theorem}
\theoremstyle{plain}

\numberwithin{equation}{section}
\input{tcilatex}

\begin{document}
\title[The detailed proof of theorem]{The detailed proof of theorem which
characterizes a slant helix}
\author{Fatih Do\u{g}an}
\address{\textit{Current Adress: Fatih} \textit{DO\u{G}AN}, \textit{Ankara
University, Department of Mathematics, 06100 Tando\u{g}an, Ankara, Turkey}}
\email{mathfdogan@hotmail.com}
\date{}
\subjclass[2000]{53A04}
\keywords{Slant helix; Principal normal indicatrix; The axis; Geodesic
curvature}

\begin{abstract}
In this paper, firstly the axis of a slant helix is found with a method.
Secondly, the theorem which characterizes a unit speed curve to be a slant
helix is proved in detail. The importance of this theorem is stemed from
that it has led to many papers regarding slant helices in the differential
geometry literature.
\end{abstract}

\maketitle

\section{Introduction}

Slant helices are special space curves like general helices, Salkowski
curves, Bertrand curves and Mannheim curves. Slant helix concept was first
introduced by Izumiya and Takeuchi and then studied by so many authors.

Izumiya and Takeuchi $[4]$ defined slant helices which are generalizations
of the notion of general helices. Kula and Yayl\i \ $[5]$ investigated the
spherical indicatrices of slant helices and showed that tangent and binormal
indicatrices of them are spherical helices. Also, Kula $et.al.$ $[6]$ gave
some characterizations for a unit speed curve $\gamma $ in $\mathbb{E}^{3}$
to be a slant helix by applying its tangent indicatrix, principal normal
indicatrix and binormal indicatrix. Ali and Turgut $[2]$ extended the notion
of slant helix from Euclidean 3-space to Euclidean n-space. They introduced
type-2 harmonic curvatures of a regular curve and by using this give
necessary and sufficient conditions for a curve to be a slant helix in
Euclidean n-space.

Recently, Ali and Turgut $[1]$ researched position vector of a timelike
slant helix in Minkowski 3-space $\mathbb{E}_{1}^{3}$. They determined the
parametric representation of the position vector $\psi $ from intrinsic
equations in $\mathbb{E}_{1}^{3}$ for the timelike slant helix. More
recently, Ali and Lopez $[3]$ looked into slant helices in Minkowski
3-space. They gave characterizations for spacelike, timelike and lightlike
slant helices, also investigated tangent indicatrix, binormal indicatrix and
involutes of non-null curve.

Monterde $[7]$ studied Salkowski curves and characterized them such that
principal normal vector makes a constant angle with a fixed line.

This paper is organized as follows. Section 2 presents basic concepts
concerning curve and surface theory in $\mathbb{E}^{3}$. Firstly, the axis
of a slant helix is found with a method and secondly, the theorem which
characterizes a unit speed curve to be a slant helix is proved in detail.

\section{Preliminaries}

In this section, we give some basic notions about curves and surfaces. The
differential geometry of curves starts with a smooth map of $s$, let us call 
$\alpha :I\subset 
\mathbb{R}
\longrightarrow \mathbb{E}^{3}$ that parametrized a spatial curve denoted
again with $\alpha $. We say that the curve is parametrized by arc-lenght if 
$\left \Vert \alpha ^{^{\prime }}(s)\right \Vert =1$ (unit speed), where $%
\alpha ^{^{\prime }}(s)$ is the first derivative of $\alpha $ with respect
to $s$. Let $\alpha :I\subset 
\mathbb{R}
\longrightarrow \mathbb{E}^{3}$ be a regular curve with an arc-lenght
parameter $s$ and $\kappa (s)=\left \Vert \alpha ^{^{\prime \prime
}}(s)\right \Vert >0$, where $\kappa $ is the curvature of $\alpha $ and $%
\alpha ^{^{\prime \prime }}$ is the second derivative of $\alpha $ with
respect to $s$. Since the curvature $\kappa $ is nonzero, the Frenet frame $%
\{T,$ $N,$ $B\}$ is well-defined along the curve $\alpha $ and as follows.%
\begin{gather}
T(s)=\alpha ^{^{\prime }}(s)  \tag{2.1} \\
N(s)=\frac{\alpha ^{^{\prime \prime }}(s)}{\left \Vert \alpha ^{^{\prime
\prime }}(s)\right \Vert }  \notag \\
B(s)=T(s)\times N(s),  \notag
\end{gather}%
where $T$, $N$ and $B$ are the tangent, the principal normal and the
binormal of $\alpha $, respectively. For a unit speed curve with $\kappa >0$
the derivatives of the Frenet frame (Frenet equations) are given by%
\begin{gather}
T^{^{\prime }}(s)=\kappa (s)N(s)  \tag{2.2} \\
N^{^{\prime }}(s)=-\kappa (s)T(s)+\tau (s)B(s)  \notag \\
B^{^{\prime }}(s)=-\tau (s)N(s),  \notag
\end{gather}%
where $\tau (s)=\dfrac{\left \langle \alpha ^{^{\prime }}(s)\times \alpha
^{^{\prime \prime }}(s),\alpha ^{^{\prime \prime \prime }}(s)\right \rangle 
}{\kappa ^{2}}$ is the torsion of $\alpha $; "$\times $" is the cross
product on $%
\mathbb{R}
^{3}$.

Let $\mathbb{M}$ be a regular surface and $\alpha :I\subset 
\mathbb{R}
\longrightarrow \mathbb{M}$ be a unit speed curve. Then, the Darboux frame $%
\{T,$ $Y=U\times T,$ $U\}$ is well-defined along the curve $\alpha $, where $%
T$ is the tangent of $\alpha $ and $U$ is the unit normal of $\mathbb{M}$.
Darboux equations of this frame are given by%
\begin{gather}
T^{^{\prime }}=k_{g}Y+k_{n}U  \tag{2.3} \\
Y^{^{\prime }}=-k_{g}T+\tau _{g}U  \notag \\
U^{^{\prime }}=-k_{n}T-\tau _{g}Y,  \notag
\end{gather}%
where $k_{n}$, $k_{g}$ and $\tau _{g}$ are the normal curvature, the
geodesic curvature and the geodesic torsion of $\alpha $, respectively. With
the above notations, let $\phi $ denote the angle between the surface normal 
$U$ and the binormal $B$. Using equations in (2.3) we reach%
\begin{gather}
\kappa ^{2}=k_{g}^{2}+k_{n}^{2}  \tag{2.4} \\
k_{g}=\kappa \cos \phi  \notag \\
k_{n}=\kappa \sin \phi  \notag \\
\tau _{g}=\tau -\phi ^{^{\prime }}.  \notag
\end{gather}%
If we rotate the Darboux frame $\{T,$ $Y=U\times T,$ $U\}$ by $\phi $ about $%
T$, we obtain the Frenet frame $\{T,$ $N,$ $B\}$.%
\begin{equation*}
\left[ 
\begin{array}{c}
T \\ 
N \\ 
B%
\end{array}%
\right] =\left[ 
\begin{array}{ccc}
1 & 0 & 0 \\ 
0 & \cos \phi & \sin \phi \\ 
0 & -\sin \phi & \cos \phi%
\end{array}%
\right] \left[ 
\begin{array}{c}
T \\ 
Y \\ 
U%
\end{array}%
\right]
\end{equation*}%
\begin{eqnarray*}
T &=&T \\
N &=&\cos \phi Y+\sin \phi U \\
B &=&-\sin \phi Y+\cos \phi U.
\end{eqnarray*}%
From the above equations, we obtain%
\begin{gather}
Y=\cos \phi N-\sin \phi B  \tag{2.5} \\
U=\sin \phi N+\cos \phi B  \notag
\end{gather}%
Let $\alpha :I\subset 
\mathbb{R}
\longrightarrow \mathbb{E}^{3}$ be a regular curve with an arc-lenght
parameter $s$ and $\kappa >0$. Since $\left \Vert N\right \Vert =1$, $\beta
:I\subset \mathbb{R}\longrightarrow \mathbb{S}^{2}$, $\beta (s)=N(s)$ along
the curve $\alpha $ lies on the unit sphere $\mathbb{S}^{2}$. The curve $%
\beta $ is called the principal normal indicatrix of $\alpha $.

\section{A Slant helix and its axis}

In this section, we find the fixed vector (axis) of a slant helix with a
method and prove theorem which characterizes a unit speed curve to be a
slant helix in detail. A space curve which its principal normal vectors make
a constant angle with a fixed vector is called a slant helix. Let $\alpha
:I\subset 
\mathbb{R}
\longrightarrow \mathbb{E}^{3}$ be a unit speed slant helix. Then, by the
definition of slant helix%
\begin{equation}
\left \langle N,d\right \rangle =\cos \theta ,  \tag{3.1}
\end{equation}%
where $N$ is the principal normal of $\alpha $, $d$ is the fixed vector of $%
\alpha $ and $\theta $ is the constant angle between $N$ and $d$. If we
differentiate Eq.(3.1) with respect to $s$ along the curve $\alpha $ and
then use Frenet equations, we obtain%
\begin{equation}
\left \langle N^{^{\prime }},d\right \rangle =0  \tag{3.2}
\end{equation}%
\begin{equation*}
\left \langle -\kappa T+\tau B,d\right \rangle =0
\end{equation*}%
\begin{equation*}
\kappa \left \langle T,d\right \rangle =\tau \left \langle B,d\right \rangle 
\end{equation*}%
\begin{equation*}
\left \langle T,d\right \rangle =\frac{\tau }{\kappa }\left \langle
B,d\right \rangle .
\end{equation*}%
If we take $\left \langle B,d\right \rangle =c$, we get%
\begin{equation*}
d=\frac{\tau }{\kappa }cT+\cos \theta N+cB.
\end{equation*}%
Since $\left \Vert d\right \Vert =1$, it follows that%
\begin{gather*}
\frac{\tau ^{2}}{\kappa ^{2}}c^{2}+\cos ^{2}\theta +c^{2}=1 \\
(\frac{\tau ^{2}}{\kappa ^{2}}+1)c^{2}=\sin ^{2}\theta  \\
c=\mp \frac{\kappa }{\sqrt{\kappa ^{2}+\tau ^{2}}}\sin \theta .
\end{gather*}%
Therefore, $d$ can be written as%
\begin{equation}
d=\mp \frac{\tau }{\sqrt{\kappa ^{2}+\tau ^{2}}}\sin \theta T+\cos \theta
N\mp \frac{\kappa }{\sqrt{\kappa ^{2}+\tau ^{2}}}\sin \theta B.  \tag{3.3}
\end{equation}%
Actually, $d$ is a constant vector. By differentiating $N^{^{\prime }}$ and
Eq.(3.2) with respect to $s$ along the curve $\alpha $, we obtain%
\begin{gather}
\left \langle N^{^{^{\prime \prime }}},d\right \rangle =0  \notag \\
\left \langle \kappa ^{^{\prime }}T+(\kappa ^{2}+\tau ^{2})N-\tau ^{^{\prime
}}B,\mp \frac{\tau }{\sqrt{\kappa ^{2}+\tau ^{2}}}\sin \theta T+\cos \theta
N\mp \frac{\kappa }{\sqrt{\kappa ^{2}+\tau ^{2}}}\sin \theta B\right \rangle
=0  \notag \\
\mp \frac{\kappa ^{^{\prime }}\tau -\tau ^{^{\prime }}\kappa }{(\kappa
^{2}+\tau ^{2})^{3/2}}\tan \theta +1=0  \notag \\
\tan \theta =\mp \frac{(\kappa ^{2}+\tau ^{2})^{3/2}}{\tau ^{^{\prime
}}\kappa -\kappa ^{^{\prime }}\tau }  \notag \\
\cot \theta =\mp \frac{\kappa ^{2}}{(\kappa ^{2}+\tau ^{2})^{3/2}}(\frac{%
\tau }{\kappa })^{^{\prime }}.  \tag{3.4}
\end{gather}%
From Eq.(3.3) and Frenet equations, the derivative of $d$ becomes%
\begin{eqnarray*}
d^{^{\prime }} &=&\mp \left[ \sin \theta (\frac{\tau }{\sqrt{\kappa
^{2}+\tau ^{2}}})^{^{\prime }}T+\sin \theta \frac{\tau }{\sqrt{\kappa
^{2}+\tau ^{2}}}(\kappa N)\right] +\cos \theta \left( -\kappa T+\tau
B\right)  \\
&&\mp \left[ \sin \theta (\frac{\kappa }{\sqrt{\kappa ^{2}+\tau ^{2}}}%
)^{^{\prime }}B+\sin \theta \frac{\kappa }{\sqrt{\kappa ^{2}+\tau ^{2}}}%
(-\tau N)\right]  \\
&=&\left[ -\kappa \mp \tan \theta (\frac{\tau }{\sqrt{\kappa ^{2}+\tau ^{2}}}%
)^{^{\prime }}\right] T+\left[ \tau \mp \tan \theta (\frac{\kappa }{\sqrt{%
\kappa ^{2}+\tau ^{2}}})^{^{\prime }}\right] B.
\end{eqnarray*}%
If we substitute Eq.(3.4) in the last equation above, we get%
\begin{eqnarray*}
d^{^{\prime }} &=&\left[ -\kappa +\frac{(\kappa ^{2}+\tau ^{2})^{3/2}}{\tau
^{^{\prime }}\kappa -\kappa ^{^{\prime }}\tau }\left( \frac{\tau ^{^{\prime
}}(\kappa ^{2}+\tau ^{2})-\tau (\kappa \kappa ^{^{\prime }}+\tau \tau
^{^{\prime }})}{(\kappa ^{2}+\tau ^{2})^{3/2}}\right) \right] T \\
&&+\left[ \tau +\frac{(\kappa ^{2}+\tau ^{2})^{3/2}}{\tau ^{^{\prime
}}\kappa -\kappa ^{^{\prime }}\tau }\left( \frac{\kappa ^{^{\prime }}(\kappa
^{2}+\tau ^{2})-\kappa (\kappa \kappa ^{^{\prime }}+\tau \tau ^{^{\prime }})%
}{(\kappa ^{2}+\tau ^{2})^{3/2}}\right) \right] B.
\end{eqnarray*}%
By a direct calculation, it can be seen that the coefficients of $T$ and $B$
are zero. Then $d^{^{\prime }}=0$ in other words $d$ is a constant vector.

\begin{theorem}
A unit speed curve $\alpha :I\subset 
\mathbb{R}
\longrightarrow \mathbb{E}^{3}$ is a slant helix if and only if%
\begin{equation*}
\sigma (s)=\mp \left( \frac{\kappa ^{2}}{(\kappa ^{2}+\tau ^{2})^{3/2}}(%
\frac{\tau }{\kappa })^{^{\prime }}\right) (s)
\end{equation*}%
is a constant function $[4]$.
\end{theorem}

\begin{proof}
The vectors make a constant angle with a fixed vector construct a cone. Then
the unit vectors in $\mathbb{E}^{3}$ make a constant angle with a fixed
vector construct a circular cone that base curve lies on the unit sphere $%
\mathbb{S}^{2}$. Therefore, $\alpha $ is the unit speed slant helix if and
only if its principal normal indicatrix is a circle on the unit sphere $%
\mathbb{S}^{2}$. In other words, if we compute the normal indicatrix $\beta
:I\subset \mathbb{R}\longrightarrow \mathbb{S}^{2}$, $\beta (s)=N(s)$ along
the curve $\alpha $, the geodesic curvature of $\beta $ becomes $\sigma (s)$
as obtained below.%
\begin{eqnarray*}
N^{^{\prime }} &=&-\kappa T+\tau B \\
N^{^{\prime \prime }} &=&-\kappa ^{^{\prime }}T-(\kappa ^{2}+\tau
^{2})N+\tau ^{^{\prime }}B \\
N^{^{\prime }}\times N^{^{\prime \prime }} &=&\tau (\kappa ^{2}+\tau
^{2})T+\kappa ^{2}(\frac{\tau }{\kappa })^{^{\prime }}N+\kappa (\kappa
^{2}+\tau ^{2})B \\
\kappa _{\beta } &=&\frac{\left \Vert N^{^{\prime }}\times N^{^{\prime
\prime }}\right \Vert }{\left \Vert N^{^{\prime }}\right \Vert ^{3}} \\
&=&\sqrt{\dfrac{(\kappa ^{2}+\tau ^{2})^{3}+(\kappa ^{2}(\dfrac{\tau }{%
\kappa })^{^{\prime }})^{2}}{(\kappa ^{2}+\tau ^{2})^{3}}},
\end{eqnarray*}%
where $\kappa _{\beta }$ is the curvature of $\beta $. Let $k_{g}$ and $%
k_{n} $ be the geodesic curvature and normal curvature of $\beta $ on $%
\mathbb{S}^{2}$, respectively. Since the normal curvature $k_{n}=1$ on $%
\mathbb{S}^{2}$, if we substitute $k_{n}$ and $\kappa _{\beta }$ in the
following equation, we get the geodesic curvature $k_{g}$ as follows.%
\begin{gather*}
k_{g}^{2}+k_{n}^{2}=(\kappa _{\beta })^{2} \\
k_{g}(s)=\sigma (s)=\cot \theta =\mp \left( \frac{\kappa ^{2}}{(\kappa
^{2}+\tau ^{2})^{3/2}}(\frac{\tau }{\kappa })^{^{\prime }}\right) (s),
\end{gather*}%
where $\theta $ is the constant angle between the principal normal $N$ and
the axis $d$. In that case, a unit speed curve $\alpha :I\subset 
\mathbb{R}
\longrightarrow \mathbb{E}^{3}$ is a slant helix if and only if the
spherical image of the principal normal $N:I\subset \mathbb{R}%
\longrightarrow \mathbb{S}^{2}$ is a circle.
\end{proof}

\end{document}